\documentclass[12pt]{amsart}
\usepackage{amssymb,amsmath,amsfonts,latexsym}
\usepackage{bm}

\usepackage{array,graphics,color}
\newcommand{\newsection}[1]{\setcounter{equation}{0} \section{#1}}
\numberwithin{equation}{section}

\newtheorem{propn}{Proposition}[section]
\newtheorem{thm}[propn]{Theorem}

\newtheorem{cor}[propn]{Corollary}
\newtheorem*{thm*}{Theorem}
\theoremstyle{definition}

\newcommand{\Hil}{\mathcal{H}}

 \newcommand{\D}{\mathbb{D}}

\newcommand{\vp}{\varphi}

\newcommand{\clb}{\mathcal{B}}

\newcommand{\cld}{\mathcal{D}}
\newcommand{\cle}{\mathcal{E}}

\newcommand{\clh}{\mathcal{H}}
\newcommand{\clk}{\mathcal{K}}

\newcommand{\clq}{\mathcal{Q}}

\newcommand{\raro}{\rightarrow}
\newcommand{\NI}{\noindent}

\begin{document}

\title{Factorizations of Contractions}

\author[Das] {B. Krishna Das}
\address{Department of Mathematics, Indian Institute of Technology Bombay, Powai, Mumbai, 400076, India}
\email{dasb@math.iitb.ac.in, bata436@gmail.com}

\author[Sarkar]{Jaydeb Sarkar}
\address{Indian Statistical Institute, Statistics and Mathematics Unit, 8th Mile, Mysore Road, Bangalore, 560059, India}
\email{jay@isibang.ac.in, jaydeb@gmail.com}

\author[Sarkar]{Srijan Sarkar}
\address{Indian Statistical Institute, Statistics and Mathematics Unit, 8th Mile, Mysore Road, Bangalore, 560059, India}
\email{srijan\_rs@isibang.ac.in, srijansarkar@gmail.com}

\dedicatory{Dedicated to Professor Rajendra Bhatia on the occasion
of his 65th birthday}



\begin{abstract}
The celebrated Sz.-Nagy and Foias theorem asserts that every pure
contraction is unitarily equivalent to an operator of the form
$P_{\mathcal{Q}} M_z|_{\mathcal{Q}}$ where $\mathcal{Q}$ is a
$M_z^*$-invariant subspace of a $\mathcal{D}$-valued Hardy space
$H^2_{\mathcal{D}}(\mathbb{D})$, for some Hilbert space
$\mathcal{D}$.

On the other hand, the celebrated theorem of Berger, Coburn and
Lebow on pairs of commuting isometries can be formulated as follows:
a pure isometry $V$ on a Hilbert space $\mathcal{H}$ is a product of
two commuting isometries $V_1$ and $V_2$ in
$\mathcal{B}(\mathcal{H})$ if and only if there exist a Hilbert
space $\mathcal{E}$, a unitary $U$ in $\mathcal{B}(\mathcal{E})$ and
an orthogonal projection $P$ in $\mathcal{B}(\mathcal{E})$ such that
$(V, V_1, V_2)$ and $(M_z, M_{\Phi}, M_{\Psi})$ on
$H^2_{\mathcal{E}}(\mathbb{D})$ are unitarily equivalent, where
\[
\Phi(z)=(P+zP^{\perp})U^*\quad \text{and} \quad
\Psi(z)=U(P^{\perp}+zP) \quad \quad (z \in \D).
\]

In this context, it is natural to ask  whether similar factorization
results hold true for pure contractions. The purpose of this paper is
to answer this question. More particularly, let $T$ be a pure
contraction on a Hilbert space $\mathcal{H}$ and let
$P_{\mathcal{Q}} M_z|_{\mathcal{Q}}$ be the Sz.-Nagy and Foias
representation of $T$ for some canonical $\mathcal{Q} \subseteq
H^2_{\mathcal{D}}(\mathbb{D})$. Then $T = T_1 T_2$, for some
commuting contractions $T_1$ and $T_2$ on $\mathcal{H}$, if and only
if there exist $\mathcal{B}(\mathcal{D})$-valued polynomials
$\varphi$ and $\psi$ of degree $ \leq 1$ such that $\mathcal{Q}$ is
a joint $(M_{\varphi}^*, M_{\psi}^*)$-invariant subspace,
\[
P_{\mathcal{Q}} M_z|_{\mathcal{Q}} = P_{\mathcal{Q}} M_{\varphi
\psi}|_{\mathcal{Q}} = P_{\mathcal{Q}} M_{\psi
\varphi}|_{\mathcal{Q}} \mbox{~and~} (T_1, T_2) \cong
(P_{\mathcal{Q}} M_{\varphi}|_{\mathcal{Q}}, P_{\mathcal{Q}}
M_{\psi}|_{\mathcal{Q}}).
\]
Moreover, there exist a Hilbert space $\mathcal{E}$ and an isometry
$V \in \mathcal{B}(\mathcal{D}; \mathcal{E})$ such that
\[
\varphi(z) = V^* \Phi(z) V \mbox{~and~} \psi(z) =
V^* \Psi(z) V \quad \quad (z \in \mathbb{D}),
\]
where the pair $(\Phi, \Psi)$, as defined above, is the Berger,
Coburn and Lebow representation of a pure pair of commuting
isometries on $H^2_{\cle}(\D)$. As an application, we obtain a
sharper von Neumann inequality for commuting pairs of contractions.
\end{abstract}
\maketitle

\newsection{Introduction}

Let $\clh$ be a Hilbert space and $V$ be an isometry on $\clh$. It
is a classical result, due to von Neumann and Wold (cf. \cite{NF}),
that $V$ is unitarily equivalent to $M_z \oplus U$ where $M_z$ is
the shift operator on an $\cle$-valued Hardy space $H^2_{\cle}(\D)$,
for some Hilbert space $\cle$, and $U$ is a unitary operator on
$\clh_u$, where
\[
\clh_u = \mathop{\cap}_{m=0}^\infty V^m \clh.
\]
We say that $V$ is \textit{pure} if $\clh_u = \{0\}$, or,
equivalently, if $V^{*m} \raro 0$ in the strong operator topology
(that is, $\|V^{*m} h\| \raro 0$ as $m \raro \infty$ for all $h \in
\clh$). Pure isometry, that is, shift operators on vector-valued
Hardy spaces play an important role in the study of general
operators that stems from the following result (see \cite{NF,
Nagy}):

\begin{thm}\label{NF}(\textsf{Sz.-Nagy and Foias})
Let $T$ be a pure contraction on a Hilbert space $\clh$. Then
$T$ and $P_{\clq} M_z|_{\clq}$ are unitarily equivalent, where
$\clq$ is a closed $M_z^*$-invariant subspace of a vector-valued
Hardy space $H^2_{\cld}(\D)$.
\end{thm}

Here the $\cld$-valued Hardy space over $\D$, denoted by
$H^2_{\cld}(\D)$, is defined by
\[
H^2_{\cld}(\D):= \{ f = \sum_{{k} \in \mathbb{N}} \eta_{{k}} z^{{k}}
\in \mathcal{O}(\D, \cld): \eta_j \in \cld, j \in \mathbb{N},
\|f\|^2 : = \sum_{{k} \in \mathbb{N}} \|\eta_{{k}}\|^2 < \infty\}.
\]
Recall that a contraction $T$ on a Hilbert space $\clh$ is
\textit{pure} (cf. \cite{JS}) if $T^{*m} \raro 0$ as $m \raro
\infty$ in the strong operator topology. Also note that, in the
above theorem, one can choose the coefficient Hilbert space
$\cld$ as $\overline{\mbox{ran}} (I - T T^*)$ (see \cite{NF}).

In contrast with the von-Neumann and Wold decomposition theorem for
isometries, the structure of commuting $n$-tuples of isometries, $n
\geq 2$, is much more complicated and very little, in general, is
known (see \cite{BDF1, BDF2, BDF3, BCL, BK, GY, Yang, Y1, Y2, Ru}).
However, for pure pairs of commuting isometries, the problem is more
tractable.

A pair of commuting isometries $(V_1, V_2)$ on a Hilbert space
$\clh$ is said to be \textit{pure} if $V_1V_2$ is a pure isometry,
that is,
\[
\mathop{\cap}_{m=0}^\infty V_1^{m} V_2^{m} \clh = \{0\}.
\]
With this as motivation, a pair of commuting contractions
$(T_1,T_2)$ is said to be \textit{pure} if $T_1T_2$ is a pure
contraction.

The concept of pure pair of commuting isometries introduced by
Berger, Coburn and Lebow \cite{BCL} is an important development in
the study of representation and Fredholm theory for $C^*$-algebras
generated by commuting isometries. They showed that a pair of
commuting isometries $(V_1, V_2)$ on a Hilbert space $\clh$ is pure
if and only if there exist a Hilbert space $\cle$, a unitary $U$ in
$\clb(\cle)$ and an orthogonal projection in $\clb(\cle)$ such that
$(V_1, V_2)$ on $\clh$ and $(M_{\Phi}, M_{\Psi})$ on
$H^2_{\cle}(\D)$ are jointly unitarily equivalent, where
\begin{equation}\label{PP}
\Phi(z)=(P+zP^{\perp})U^*\quad \text{and} \quad
\Psi(z)=U(P^{\perp}+zP) \quad \quad (z \in \D).
\end{equation}
Moreover, it follows that
\[
M_{\Phi} M_{\Psi} = M_{\Psi} M_{\Phi} = M_z,
\]
and $V_1V_2$ on $\clh$ and $M_z$ on $H^2_{\cle}(\D)$ are unitarily
equivalent (see also \cite{BDF3, GG}). More precisely, if $\Pi :
\clh \raro H^2_{\cle}(\D)$ denotes the unitary map, implemented by
the Wold and von Neumann decomposition of the pure isometry $V_1
V_2$ with $\cle = \mbox{ran} (I - V_1 V_2 V_1^* V_2^*)$ (cf.
\cite{JS}), then
\[
\Pi V_1 = M_{\Phi} \Pi, \quad \mbox{~and~} \quad  \Pi V_2 = M_{\Psi}
\Pi.
\]

\textsf{In what follows, for a triple $(\cle, U, P)$ as above we let
$\Phi, \Psi \in H^\infty_{\clb(\cle)}(\D)$ denote the isometric
multipliers as defined in (\ref{PP}). We call
 $(M_{\Phi}, M_{\Psi})$ the pair of isometries associated with the triple $(\cle, U,
P)$.}

Our work is motivated by the following equivalent interpretations of
the Berger, Coburn and Lebow's characterizations of pure pairs of
commuting isometries:

\NI \textsf{($I$)} Let $(V_1, V_2)$ be a pure pair of commuting
isometries and let $M_z$ on $H^2_{\cle}(\D)$ be the von Neumann and
Wold decomposition representation of $V_1V_2$. Then there exist a
unitary $U$ and an orthogonal projection $P$ in $\clb(\cle)$ such
that the representations of $V_1$ and $V_2$ in
$\clb(H^2_{\cle}(\D))$ are given by $M_{\Phi}$ and $M_{\Psi}$,
respectively.

\NI \textsf{($II$)} Let $(X, Y)$ be a pair of commuting isometries
in $\clb(H^2_{\cle}(\D))$. Moreover, let $X$ and $Y$ are Toeplitz
operators \cite{NF} with analytic symbols from
$H^\infty_{\clb(\cle)}(\D)$. Then $M_z = XY$ if and only if there
exist a unitary $U$ and an orthogonal projection $P$ in $\clb(\cle)$
such that $(X, Y) = (M_{\Phi}, M_{\Psi})$.

In this paper we shall obtain similar results for pure pairs of
commuting contractions acting on Hilbert spaces. More specifically,
summarizing Theorems \ref{main1}, \ref{main2} and \ref{factor-Q} and
Corollary \ref{cor1}, we have the following: Let $T$ be a pure
contraction and $(T_1, T_2)$ be a pair of commuting contractions on
a Hilbert space $\clh$. Let $\clq$ be the Sz.-Nagy and Foias
representation of $T$, that is, $\clq$ is a $M_z^*$-invariant
subspace of a vector-valued Hardy space $H^2_{\cld}(\D)$ and $T$ and
$P_{\clq} M_z|_{\clq}$ are unitarily equivalent (see Theorem
\ref{NF} and Section \ref{prel}) where $\cld = \overline{\mbox{ran}}
(I - T T^*)$. Then the following are equivalent:

(i) $T = T_1 T_2$.

(ii) There exist a triple $(\cle, U, P)$ and a joint $(M_z^*,
M_{\Phi}^*,M_{\Psi}^*)$-invariant subspace $\tilde{\clq}$ of
$H^2_{\cle}(\D)$ such that
\[
T_1\cong P_{\tilde{\clq}}M_{\Phi}|_{\tilde{\clq}}, T_2\cong
P_{\tilde{\clq}}M_{\Psi}|_{\tilde{\clq}}, T \cong P_{\tilde{\clq}}
M_z|_{\tilde{\clq}} \text{~and~} M_{\Phi} M_{\Psi} = M_{\Psi}
M_{\Phi} = M_z.
\]
In other words, $(T_1, T_2, T)$ on $\clh$ dilates to $(M_{\Phi},
M_{\Psi}, M_z)$ on $H^2_{\cle}(\D)$.

(iii) There exist $\clb(\cld)$-valued polynomials $\vp$ and $\psi$
of degree $ \leq 1$ such that $\clq$ is a joint $(M_{\vp}^*,
M_{\psi}^*)$-invariant subspace,
\[
P_{\clq} M_z|_{\clq} = P_{\clq} M_{\vp \psi}|_{\clq} = P_{\clq}
M_{\psi \vp}|_{\clq},
\]
and
\[
(T_1, T_2) \cong (P_{\clq} M_{\vp}|_{\clq}, P_{\clq}
M_{\psi}|_{\clq}).
\]

In particular, if $T = T_1 T_2$ is pure then the Sz.-Nagy and Foias
representations of $T_1$ and $T_2$ on $\clq$ are given by $P_{\clq}
M_{\vp}|_{\clq}$ and $P_{\clq} M_{\psi}|_{\clq}$, respectively.
Moreover, it turns out that the pair $(M_{\vp}, M_{\psi})$ can be chosen as
\[
\vp(z) = V^* \Phi(z) V \mbox{~and~} \psi(z) =
V^* \Psi(z) V \quad \quad (z \in \D),
\]
where  $V \in \clb(\cld; \cle)$ is an isometry and $\Phi$, $\Psi$ are
the isometric multipliers associated with the triple $(\cle, U, P)$
from condition (ii).
As an application of our results we give a sharper von Neumann
inequality for commuting pairs of contractions: Let $(T_1,T_2)$ be a
commuting pair of contractions on $\Hil$. Also assume that $T_1
T_2$ is a pure contraction and $\mbox{rank~}(I_{\clh} - T_i
T_i^*)<\infty$, $i=1,2$. Then there exists a variety $V$ in $\D^2$
such that
\[
\|p(T_1,T_2)\|\le \sup_{(z_1,z_2)\in  V}|p(z_1,z_2)| \quad \quad (p
\in \mathbb{C}[z_1, z_2]).
\]

The plan of the paper is the following. Section 2 contains some
preliminaries and a key dilation result. In Section \ref{DPI}, we
prove that a pure pair of commuting contractions always dilates to a
pure pair of commuting isometries. Our construction is more explicit
for pairs of contractions with finite dimensional defect spaces. In
Section \ref{Fact}, we obtain explicit representations of commuting
and contractive factors of a pure contraction in its corresponding
Sz.-Nagy and Foias space. In the last section, we consider von
Neumann inequality for pure pair of commuting contractions.

\newsection{Preliminaries}\label{prel}

In this section, we set notation and definitions and discuss some
preliminaries. Also we prove a basic dilation result in Theorem
\ref{dilation}. This result will play a fundamental role throughout
the remainder of the paper.

Let $T$ be a contraction on a Hilbert space $\clh$ (that is, $\|T f
\| \leq \|f\|$ for all $f \in \clh$ or, equivalently, if $I_{\clh} -
T T^* \geq 0$) and let $\cle$ be a Hilbert space. Then $M_z$ on
$H^2_{\cle}(\D)$ is called an isometric dilation of $T$ (cf.
\cite{JS}) if there exists an isometry $\Gamma : \clh \raro
H^2_{\cle}(\D)$ such that
\[
\Gamma T^* = M_z^* \Gamma.
\]
Similarly, a pair of commuting operators $(U_1, U_2)$ on $\clk$ is
said to be a dilation of a commuting pair of operators $(T_1, T_2)$
on $\clh$ if there exists an isometry $\Gamma : \clh \raro \clk$
such that
\[
\Gamma T_j^* = U_j^* \Gamma \quad \quad (j = 1, 2).
\]
Note that, in this case, $\clq := \mbox{ran} \Gamma$ is a joint
$(U_1^*, U_2^*)$-invariant subspace of $\clk$ and
\[
T_j \cong P_{\clq} U_j|_{\clq} \quad \quad (j=1, 2).
\]

Now let $T$ be a contraction on a Hilbert space $\clh$. Set
\[
\cld_T = \overline{\mbox{ran}}(I_{\clh} - T T^*), \quad \quad D_T =
(I_{\clh} - T T^*)^{\frac{1}{2}}.
\]
If in addition, $T$ is pure then $M_z$ on $H^2_{\cld_T}(\D)$,
induced by the isometry $\Pi : \clh \raro H^2_{\cld_T}(\D)$, is an
isometric dilation of $T$ (cf. \cite{JS}), where
\begin{equation}\label{dil-def}
(\Pi h)(z) = D_T(I_{\clh} - z T^*)^{-1}h \quad \quad (z \in \D, h
\in \clh).
\end{equation}
In particular, this yields a proof of Theorem \ref{NF} that every
pure contraction is unitarily equivalent to the compression of $M_z$
to an $M_z^*$-invariant closed subspace of a vector-valued Hardy
space.

It is also important to note that the above dilation is minimal,
that is,
\begin{equation}\label{min}
H^2_{\cld_T}(\D) =
\overline{\mbox{span}} \{z^m \Pi f : m \in \mathbb{N}, f \in \clh\},
\end{equation}
and hence unique in an appropriate sense (see \cite{NF}).
%and also the
%Factorization theorem, Theorem 4.1, in \cite{JS3}).

Our considerations will also rely on the techniques of transfer
functions (cf. \cite{DS}). Let $\clh_1$ and $\clh_2$ be two
 Hilbert spaces, and
\[U = \begin{bmatrix}A&B\\C&D\end{bmatrix} \in \clb(\clh_1 \oplus
\clh_2),\]be a unitary operator. Then the $\clb(\clh_1)$-valued
analytic function $\tau_U$ on $\mathbb{D}$ defined by
\[\tau_U (z) := A + z B (I-z D)^{-1} C \quad \quad (z \in \D), \]
is called the \textit{transfer function} of $U$. Using $U^* U = I$,
a standard and well known computation yields (cf. \cite{DS})
\begin{equation}\label{isometry}
I- \tau_U (z)^*\tau_U (z) = (1-|z|^2) C^*(I-\bar{z} D^*)^{-1}(I-z
D)^{-1} C \quad \quad (z\in\D).
\end{equation}
However, in this paper, we will mostly deal with transfer functions
corresponding to unitary matrices of the form $U =
\begin{bmatrix}A&B\\C&0\end{bmatrix}$. In this case, it follows (see
(\ref{isometry})) from the identity
\[
I- \tau_U (z)^*\tau_U (z) = (1-|z|^2)C^*C \quad \quad (z \in \D)
\]
that $\tau_U$ is a $B(\clh_1)$-valued inner function \cite{NF}.

Now let $(T_1, T_2)$ be a pair of commuting contractions. Since
\[(I_{\clh} - T_1 T_1^*) + T_1 (I_{\clh} - T_2 T_2^*) T_1^* = T_2
(I_{\clh} - T_1 T_1^*) T_2^* + (I_{\clh} - T_2 T_2^*),\] it follows
that
\[ \|D_{T_1}h\|^2+
\|D_{T_2}T_1^*h\|^2=\|D_{T_1}T_2^*h\|^2+\|D_{T_2}h\|^2 \quad
(h\in\Hil).
\]
Thus \[U : \{D_{T_1} h \oplus D_{T_2} T_1^* h : h \in \clh\} \raro
\{D_{T_1} T_2^* h \oplus D_{T_2} h : h \in \clh\}\]defined by
\begin{equation}\label{U-h}
U\left(D_{T_1}h, D_{T_2}T_1^*
h\right)=\left(D_{T_1} T_2^*h, D_{T_2}h\right) \quad \quad (h \in
\clh),
\end{equation}
is an isometry. This operator will play a very important role in the
sequel.

We now formulate the main theorem of this section, a result which
will play a very important part in our considerations later on. Here
the proof is similar in spirit to the main dilation result of
\cite{DS}.

Let $\clh$ and $\cle$ be Hilbert spaces and let $(S,T)$ be a pair of
commuting contractions on $\Hil$. Let $T$ be pure and $V \in
\clb(\cld_{T}; \cle)$ be an isometry. Then the isometric dilation of
$T$, $\Pi: \Hil\to H^2_{\cld_T}(\D)$ as defined in (\ref{dil-def}),
allows us to define an isometry $\Pi_V \in \clb(\clh;
H^2_{\cle}(\D))$ by setting
\[
\Pi_V : = (I_{H^2(\mathbb{D})} \otimes V) \Pi.
\]
It is easy to check that
\[
\Pi_V T^* = (M_z^* \otimes I_{\cle}) \Pi_V,
\]
and hence we conclude that $M_z$ on $H^2_{\cle}(\D)$ is an isometric
dilation of $T$. In particular, $\clq=\Pi_V \Hil$ is a
$M_z^*$-invariant subspace of $ H^2_{\cle}(\D)$ and $T\cong
P_{\clq}M_z|_{\clq}$.

\begin{thm}\label{dilation}
With the notations as above, let
\[
U=\begin{bmatrix}A&B\\C&0\end{bmatrix}:\cle\oplus\cld_{S}\to\cle\oplus\cld_{S},
\]
be a unitary operator such that
\[U(V D_{T}h, D_{S}T^*h)=(V D_TS^*h, D_Sh) \quad \quad (h\in\Hil).
\]
We denote by $\Phi(z)=A^*+ z C^*B^*$ the transfer function of $U^*$.
Then $\Phi$ is a $\clb(\cle)$-valued inner function and
\[
\Pi_V S^* = M_{\Phi}^* \Pi_V.
\]
In particular, $\clq=\Pi_V \Hil$ is a joint $(M_z^*,
M_{\Phi}^*)$-invariant subspace of $ H^2_{\cle}(\D)$ and
\[
T^*\cong M_z^*|_{\clq}\quad \text{ and } S^*\cong
M_{\Phi}^*|_{\clq}.
\]
\end{thm}
\begin{proof}
We only need to prove that $\Pi_V S^* = M_{\Phi}^* \Pi_V$. Now for
each $h\in\Hil$ we have the equality

\[
\begin{bmatrix}A&B\\C&0\end{bmatrix}
\begin{bmatrix} V D_T h \\ D_S T^* h \end{bmatrix}
=
\begin{bmatrix} V D_T S^* h \\ D_S h \end{bmatrix},
\]
that is,
\[
V D_T S^* h = A V D_T h + B D_S T^* h, \quad\mbox{and} \quad D_S h =
C V D_T h.
\]
This implies
\[
V D_T S^* = A V D_T + B C V D_T T^*.
\]
Now if $n \geq 1$, $h\in\Hil$ and $\eta \in \cle$, then
\[
\begin{split}
\langle M_{\Phi}^* \Pi_V h, z^n \eta \rangle & = \langle(I\otimes V) D_T (I -
z T^*)^{-1} h, (A^* + z C^* B^*)(z^n \eta) \rangle
\\
& = \langle V D_T T^{*n} h, A^* \eta \rangle + \langle V D_T T^{*
(n+1)} h, C^* B^* \eta \rangle
\\
& = \langle (A V D_T + B C V D_T T^{*}) (T^{*n} h), \eta \rangle
\\
& = \langle V D_T S^* (T^{*n} h), \eta \rangle.
\end{split}
\]
On the other hand, since
\[
\langle \Pi_V S^* h, z^n \eta \rangle = \langle V D_T (I - z
T^*)^{-1} S^* h, z^n \eta \rangle = \langle (V D_T S^*) (T^{*n} h),
\eta \rangle,
\]
we get $\Pi_V S^* = M_{\Phi}^* \Pi_V$. This completes the proof.
\end{proof}

\newsection{Dilating to pure isometries}\label{DPI}

In this section we prove that a pure pair of commuting contractions
dilates to a pure pair of commuting isometries. We describe the
construction of dilations more explicitly in the case of finite
dimensional defect spaces.

\begin{thm}\label{main1}
Let $(T_1,T_2)$ be a pure pair of commuting contractions on $\Hil$
and $\mbox{dim~} \cld_{T_j} < \infty$, $j = 1, 2$. Then $(T_1,T_2)$
dilates to a pure pair of commuting isometries.
\end{thm}
\begin{proof}
Set $\cle:=\cld_{T_1}\oplus\cld_{T_2}$ and $T:=T_1T_2$. Let
$\Pi:\Hil\to H^2_{D_{T}}(\D)$ be the isometric dilation of $T$ as
defined in (\ref{dil-def}).
 Now observe that the equality
\[
I - T T^* = I - T_1 T_2 T_1^* T_2^* = (I - T_1 T_1^*) + T_1 (I - T_2
T_2^*) T_1^*,
\]
implies that the operator $V \in \clb(\cld_{T}; \cle)$ defined by
\[
V(D_T h) =
(D_{T_1}h, D_{T_2}T_1^*h)\quad \quad (h\in\Hil),
\]
is an isometry.
Consequently,

\begin{equation}\label{piv}
\Pi_V:= (I_{H^2(\D)} \otimes V) \Pi : \clh \raro H^2_{\cle}(\D)
\end{equation}
is an isometric dilation of $T$, and hence $T \cong P_{\clq}
M_z|_{\clq}$ where $\clq = \Pi_V \Hil$ is a $M_z^*$-invariant
subspace of $H^2_{\cle}(\D)$ (see the proof of Theorem
\ref{dilation}). Let $\iota_j : \cld_{T_j} \raro \cle$, $j = 1,2$, be
the inclusion maps, defined by
\[
\iota_1(h_1) = (h_1,0)\quad  \mbox{and} \quad \iota_2 (h_2) = (0,h_2)
\quad \quad (h_1 \in \cld_{T_1}, h_2 \in \cld_{T_2}).
\]
Then $P : = \iota_2 \iota_2^* \in \clb(\cle)$ is the orthogonal
projection onto $\cld_{T_2}$, that is,
\[
P(h_1,h_2) = (0,h_2) \quad \quad \quad ((h_1,h_2) \in \cle).
\]
Thus, $\iota_1 \iota_1^* = P^\perp$ is the orthogonal projection
onto $\cld_{T_1}$, and so
\[
 \begin{bmatrix} P & \iota_1\\ \iota_1^* & 0\end{bmatrix}:
 \cle \oplus \cld_{T_1}\to \cle \oplus \cld_{T_1}
\]
is a unitary. Now since $\mbox{dim~} \cle < \infty$, it follows that
the isometry $U$, as defined in (\ref{U-h}), extends to a unitary,
denoted again by $U$, on $\cle$. In particular, there exists a
unitary operator $U$ on $\cle$ such that
\[
U(D_{T_1}T_2^*h, D_{T_2}h) = (D_{T_1}h, D_{T_2}T_1^*h) \quad \quad
(h\in\Hil).
\]
Then
\[
U_1 =
\begin{bmatrix} U & 0\\0 &I
\end{bmatrix}
\begin{bmatrix}
P & \iota_1\\
\iota_1^*  & 0
\end{bmatrix}
=
\begin{bmatrix}
UP& U \iota_1\\
\iota_1^*& 0
\end{bmatrix},
\]
is a unitary operator in $\clb(\cle\oplus\cld_{T_1})$. Moreover, for
all $h\in\Hil$, we have
\[
\begin{split}
 U_1\big(V(D_Th),D_{T_1}T^*h\big)&=U_1\big(D_{T_1}h,D_{T_2}T_1^*h,D_{T_1}T_1^*T_2^*h\big)\nonumber\\
 &=\big(U(D_{T_1}T_1^*T_2^*h, D_{T_2}T_1^*h),D_{T_1}h\big)\nonumber\\
 &=\big(D_{T_1}T_1^*h,D_{T_2}(T_1^*)^2h, D_{T_1}h\big)\nonumber\\
 &=\big(V(D_{T}T_1^*h), D_{T_1}h\big).
\end{split}
\]
Consequently, by Theorem~\ref{dilation} we have
\[
\Pi_V T_1^* = M_{\Phi}^* \Pi_V,
\]
where
\[
\Phi(z)= PU^*+ z \iota_1\iota_1^*U^*= (P+zP^{\perp})U^*
\quad \quad (z \in \D),
\]
is the transfer function of the unitary operator $U_1^*$. Similarly,
if we define a unitary $U_2 \in \clb(\cle\oplus \cld_{T_2})$ by
\[
U_2 = \begin{bmatrix} P^{\perp}& \iota_2\\ \iota^*_2& 0
\end{bmatrix}
\begin{bmatrix}
U^* & 0\\
0& I
\end{bmatrix}
=\begin{bmatrix}
P^{\perp}U^* &\iota_2\\
\iota_2^*U^*& 0
\end{bmatrix},
\]
then
\[
 U_2\big(V(D_{T}h),D_{T_2}T^*h\big)=\big(V(D_{T}T_2^*h),
 D_{T_2}h\big) \quad \quad (h\in\Hil),
\]
and hence by Theorem~\ref{dilation}, we have
\[
\Pi_V T_2^* = M_{\Psi}^* \Pi_V,
\]
where
\[\Psi(z)= U P^{\perp}+ z U\iota_2\iota_2^*=U(P^{\perp}+ zP),\]
is the transfer function for the unitary operator $U_2^*$. This
completes the proof that the pure pair of commuting isometries
$(M_{\Phi}, M_{\Psi})$ on $H^2_\cle(\D)$ corresponding to the triple
$(\cle, U, P)$ dilates $(T_1,T_2)$.
\end{proof}

We will now go on to give a proof of the general result. The proof
is essentially the same as the previous theorem except the
constructions of unitary operators and inclusion maps.

\begin{thm}\label{main2}
Let $(T_1,T_2)$ be a pure pair of commuting contractions on $\Hil$.
Then $(T_1,T_2)$ dilates to a pure pair of commuting isometries.
\end{thm}
\begin{proof}
Let $\dim \cld_{T_1} = \infty$, or $\dim \cld_{T_2} = \infty$ and
$\cld$ be an infinite dimensional Hilbert space. Set
$\cle:=(\cld\oplus\cld_{T_1})\oplus\cld_{T_2}$. We now define
inclusion maps $\iota_1:\cld\oplus\cld_{T_1}\to\cle$ and
$\iota_2:\cld_{T_2}\to \cle$ by
\[
\iota_1(h,h_1) = (h,h_1,0) \quad \mbox{and} \quad \iota_2 h_2 =
(0,0,h_2),\quad (h\in\cld, h_1\in \cld_{T_1}, h_2\in\cld_{T_2})
\]
respectively, and an isometric embedding $V \in \clb(\cld_T; \cle)$
by
\[
V D_T h = (0,D_{T_1}h,D_{T_2}T_1^*h)\quad (h\in\Hil).
\]
We also define the orthogonal projection $P$ by $P = \iota_2
\iota_2^*$. Therefore
\[
P (h_1,h_2,h_3) = (0, 0,h_3) \quad \quad  ((h_1,h_2,h_3) \in \cle).
\]
Finally, since
\[
U_{\cld} \left(0_{\cld}, D_{T_1}h, D_{T_2}T_1^* h
\right)=\left(0_{\cld}, D_{T_1} T_2^*h, D_{T_2}h \right) \quad \quad
(h \in \clh),
\]
defines an isometry from $\{0_{\cld}\}\oplus\{D_{T_1} h \oplus
D_{T_2} T_1^* h : h \in \clh\}$ to $\{0_{\cld}\}\oplus\{D_{T_1}
T_2^* h \oplus D_{T_2} h : h \in \clh\}$, we can therefore extend
$U_{\cld}$ to a unitary, denoted again by $U_{\cld}$, acting on
$\cle$. With these notations we define unitary operators
\[
 U_1=\begin{bmatrix}
U_\cld P& U_\cld\iota_1\\
\iota_1^*& 0
\end{bmatrix} \in \clb(\cle\oplus(\cld\oplus\cld_{T_1})) \quad \text{and } U_2=\begin{bmatrix}
P^{\perp}U_{\cld}^* &\iota_2\\
\iota_2^*U_{\cld}^*& 0
\end{bmatrix}
\in \clb(\cle \oplus\cld_{T_2}).
\]
The rest of the proof proceeds in the same way as in Theorem
\ref{main1}. This completes the proof.
\end{proof}

The main inconvenience of our approach seems to be the nonuniqueness
of the triple $(\cle, U, P)$. This issue is closely related to the
nonuniqueness of Ando dilation \cite{An} and solutions of commutant
lifting theorem \cite{FF}.

It is also important to note that Theorem \ref{main2} is a sharper
version of Ando dilation \cite{An} for pure pairs of commuting
contractions. More precisely, one can dilate a pure pair of
commuting contractions to a pure pair of commuting isometries, in
the sense of Berger, Coburn and Lebow. In the context of concrete
isometric dilations for commuting pairs of pure contractions, see
\cite{AM1} and \cite{DS}.

\newsection{Factorizations}\label{Fact}

Let $(T_1, T_2)$ be a pair of commuting contractions on $\clh$ and
$T = T_1 T_2$ be a pure contraction. Then by Theorem \ref{NF} we can
realize $T$ as $P_{\clq} M_z|_{\clq}$ where $\clq = \mbox{ran} \Pi =
\Pi \clh$ is the Sz.-Nagy and Foias model space and $\Pi : \clh
\raro H^2_{\cld_T}(\D)$ is the minimal isometric dilation of $T$
(see (\ref{min})).

In this section we will show that $T_1$ and $T_2$ can be realized as
compressions of two $\clb(\cld_T)$-valued polynomials of degree
$\leq 1$ in the Sz.-Nagy and Foias model space $\clq$ of the pure
contraction $T$.

%Our method involves a ``pull-back'' technique (see \cite{JS3}) of
%the pure pair of isometric dilation, as obtained in Theorems
%\ref{main1} and \ref{main2}, to the Sz.-Nagy and Foias minimal
%isometric dilation.

Let $\Pi_V : \clh \raro H^2_{\cle}(\D)$ be the isometric dilation as
in Theorems \ref{main1} and \ref{main2}, that is,
\[
\Pi_{V} T_1^* = M_{\Phi}^* \Pi_V \quad \mbox{and}\quad \Pi_V T_2^* =
M_{\Psi}^* \Pi_V.
\]
%Now by canonical factorization of dilations (cf. Theorem 4.1
%\cite{JS3}), there exists a unique isometry $\tilde{V} \in
%\clb(\cld_T; \cle)$ such that the following diagram commutes
%\setlength{\unitlength}{3mm}
%\begin{center}
%\begin{picture}(40,16)(0,0)
%\put(15,3){$\clh$}\put(19,1.6){$\Pi_V$}
%\put(22.9,3){$H^2_{\cle}(\D)$} \put(22, 10){$H^2_{\cld_T}(\D)$}
%\put(22.8,9.2){ \vector(0,-1){5}} \put(15.8, 4.3){\vector(1,1){5.8}}
%\put(16.4, 3.4){\vector(1,0){6}}\put(16.5,8){$\Pi$}\put(24.3,7){$I
%\otimes \tilde{V}$}
%\end{picture}
%\end{center}
%that is
%\[
%\Pi_V = (I \otimes \tilde{V}) \Pi.
%\]
%By uniqueness of $\tilde{V}$ and the definition of $\Pi_V$ (see
%(\ref{piv})), we have that
%\[
%V = \tilde{V}.
%\]
Then it follows from ~\ref{piv} that
\[
\Pi T_1^* = (I \otimes {V}^*) M_{\Phi}^* (I \otimes {V}) \Pi =
M_{\vp}^* \Pi,
\]
where
\[
\vp(z) = {V}^* \Phi(z) {V} \quad \quad (z \in \D),
\]
and $V\in \mathcal{B}(\cld_T;\cle)$ is an isometry.
Similarly, we derive
\[
\Pi T_2^* = (I \otimes {V}^*) M_{\Psi}^* (I \otimes {V}) \Pi =
M_{\psi}^* \Pi,
\]
where
\[
\psi(z) = {V}^* \Psi(z) {V} \quad \quad (z \in \D).
\]
In particular, $\mbox{ran~} \Pi$ is a joint $(M_{\vp}^*,
M_{\psi}^*)$-invariant subspace and by construction of $\Pi$
it follows that
\[
\begin{split}
\Pi T^* = M_z^* \Pi,
\end{split}
\]
and $\mbox{ran~} \Pi$ is a also a $M_z^*$-invariant subspace of
$H^2_{\cld_T}(\D)$. We have thus proved the following theorem.

\begin{thm}\label{factor-Q}
Let $T$ be a pure contraction on a Hilbert space $\clh$ and let $T
\cong P_{\clq} M_z|_{\clq}$ be the Sz.-Nagy and Foias representation
of $T$ as in Theorem \ref{NF}. If $T = T_1 T_2$, for some commuting
pair of contractions $(T_1, T_2)$ on $\clh$, then there exists
$\clb(\cld_T)$-valued polynomials $\vp$ and $\psi$ of degree $ \leq
1$ such that $\clq$ is a joint $(M_{\vp}^*, M_{\psi}^*)$-invariant
subspace and
\[
(T_1, T_2) \cong (P_{\clq} M_{\vp}|_{\clq}, P_{\clq}
M_{\vp}|_{\clq}).
\]
In particular,
\[
P_{\clq} M_z|_{\clq} = P_{\clq} M_{\vp \psi}|_{\clq} = P_{\clq}
M_{\vp \phi}|_{\clq}.
\]
\end{thm}

It is important to note that $P_{\clq} M_{\vp \psi}|_{\clq} =
P_{\clq} M_{\psi \vp}|_{\clq}$, even though, in general
\[
\vp \psi \neq \psi \vp.
\]

A reformulation of Theorem \ref{factor-Q} is the following:

\begin{cor}\label{cor1}
Let $T$ be a pure contraction on a Hilbert space $\clh$ and let $T
\cong P_{\clq} M_z|_{\clq}$ be the Sz.-Nagy and Foias representation
of $T$ as in Theorem \ref{NF}. Then $T = T_1 T_2$, for some
commuting pair of contractions $(T_1, T_2)$ on $\clh$, if and only
if there exist $\clb(\cld_T)$-valued polynomials $\vp$ and $\psi$ of
degree $ \leq 1$ such that $\clq$ is a joint $(M_{\vp}^*,
M_{\psi}^*)$-invariant subspace,
\[
P_{\clq} M_z|_{\clq} = P_{\clq} M_{\vp \psi}|_{\clq} = P_{\clq}
M_{\psi \vp}|_{\clq},
\]
and
\[
(T_1, T_2) \cong (P_{\clq} M_{\vp}|_{\clq}, P_{\clq}
M_{\psi}|_{\clq}).
\]
Moreover, there exists a triple $(\cle, U, P)$ and an isometry ${V}
\in \clb(\cld_T; \cle)$ such that
\[
\vp(z) = {V}^* \Phi(z) {V} \mbox{~and~} \psi(z) = {V}^* \Psi(z) {V}
\quad \quad (z \in \D).
\]
\end{cor}

\newsection{von Neumann inequality}
In this section we consider the von Neumann inequality for
 pure pair of commuting contractions with finite dimensional
defect spaces. We show that for such
a pair there exists a variety in the bidisc where the von Neumann
inequality holds.

\begin{thm}\label{vn_in}
Let $(T_1,T_2)$ be a pure pair of commuting contractions on $\Hil$
and $\mbox{dim~}\cld_{T_i}<\infty$, $i=1,2$. Then there exists an algebraic
variety $V$ in $\overline{\D}^2$ such that
\[
\|p(T_1,T_2)\|\le \sup_{(z_1,z_2)\in  V}|p(z_1,z_2)| \quad \quad (p
\in \mathbb{C}[z_1, z_2]).
\]
Moreover, if $m = \dim(\cld_{T_1} \oplus \cld_{T_2})$, then there
exists a pure pair of commuting isometries $(M_{\Phi}, M_{\Psi})$ on
$H^2_{\mathbb{C}^m}(\mathbb{D})$ such that
\[
V = \{(z_1, z_2) \in \overline{\D}^2: \det(\Phi(z_1
z_2) - z_1 I) =0 \mbox{~and~} \det(\Psi(z_1 z_2)- z_2 I)=0\}.
\]
\end{thm}
\begin{proof}
Let $\cle = \cld_{T_1} \oplus \cld_{T_2}$. Then by
Theorem~\ref{main1}, there exists a pure pair of commuting isometries
$(M_{\Phi}, M_{\Psi})$ on $H^2_{\cle}(\mathbb{D})$ and a joint
$(M_{\Phi}^*, M_{\Psi}^*)$-invariant subspace $\clq$ of
$H^2_{\cle}(\mathbb{D})$ such that $T_1 \cong P_{\clq}
M_{\Phi}|_{\clq}$ and $T_2 \cong P_{\clq} M_{\Psi}|_{\clq}$. Then
for each $p\in \mathbb{C}[z_1,z_2]$, we have
\begin{align}
\label{vN} \|p(T_1,T_2)\|_{\clb(\clh)} & = \|P_{\clq}
p(M_{\Phi},M_{\Psi})|_{\clq}\|_{\clb(\clq)}
\\
& \le \|p(M_{\Phi},M_{\Psi})\|_{\clb(H^2_{\cle}(\mathbb{D}))} \nonumber \\
&= \| M_{p(\Phi, \Psi)}\|_{\clb(H^2_{\cle}(\mathbb{D}))} \nonumber\\
&\le \sup_{z\in \mathbb{T}}\|p\left(\Phi(z),\Psi(z)\right)\|_{\clb(\cle)} \nonumber\\
&\le \sup\{|p(\lambda_1,\lambda_2)|: (\lambda_1,\lambda_2)\in
\sigma(\Phi(z),\Psi(z)), z\in\mathbb{T}\},
\end{align}
where we denote $\sigma(\Phi(z),\Psi(z))$ by the joint spectrum of
the commuting pair of unitary matrices $(\Phi(z), \Psi(z))$, $z \in
\mathbb{T}$. Now observe that if $(\lambda_1,\lambda_2)\in
\sigma(\Phi(z),\Psi(z))$ for some $z\in\mathbb{T}$, then there
exists a non-zero $h\in\cle$ such that $\Phi(z)h=\lambda_1h$ and
$\Psi(z)h=\lambda_2h$. Then $zh=\Phi(z)\Psi(z)h=\lambda_1\lambda_2h$
and hence $z=\lambda_1\lambda_2$. With this observation we have
\[
\{(\lambda_1,\lambda_2)\in \sigma(\Phi(z),\Psi(z)):
z\in\mathbb{T}\}\subset \partial V,
\]
where
\[
V_1= \{(\lambda_1,\lambda_2)\in\overline{\D}^2:
\mbox{det}(\Phi(\lambda_1\lambda_2)-\lambda_1I)=0\},
\]
and
\[
V_2=\{(\lambda_1,\lambda_2)\in\overline{\D}^2:
\mbox{det}(\Psi(\lambda_1\lambda_2)-\lambda_2I)=0\},
\]
and
\[
V=V_1\cap V_2.
\]
Now since $\Phi$ and $\Psi$ are matrix valued polynomial, the
variety $V$ is an algebraic variety in $\overline{\D}^2$. Note also
that equation ~\eqref{vN} implies that
\[
\|p(T_1,T_2)\|\le \sup_{(z_1,z_2)\in \partial V}|p(z_1,z_2)|,
\]
and hence
\[
\|p(T_1,T_2)\|\le \sup_{(z_1,z_2)\in  V}|p(z_1,z_2)|,
\]
for all $p\in \mathbb{C}[z_1,z_2]$. This completes the proof.
\end{proof}

With the hypotheses of Theorem \ref{vn_in}, let $(\cle, U, P)$ be
the triple corresponding to the pure isometric pair $(M_{\Phi},
M_{\Psi})$. Let also assume that $PU^*$ and $UP^{\perp}$ be
completely non-unitary. It now follows from \cite[Proposition
4.1]{DS} that $\Phi(z)$ and $\Psi(z)$, $z\in\D$, does not have any
unimodular eigenvalue. Therefore
\[
V \cap \{ (\D\times \mathbb{T}) \cup (\mathbb{T}\times \D)
\}=\emptyset,
\]
and hence
\[
V\cap \partial \D^2= V\cap \mathbb{T}^2.
\]
This allows one to replace the algebraic variety $V$ in Theorem
~\ref{vn_in} by an algebraic distinguished variety (see \cite{AM1})
\[
\tilde{V} = \tilde{V}_1\cap\tilde{V}_2,
\]
in $\D^2$, where
\[
\tilde{V}_1= \{(\lambda_1,\lambda_2)\in\D^2:
\mbox{det}(\Phi(\lambda_1\lambda_2)-\lambda_1I)=0\},
\]
and
\[
\tilde{V}_2=\{(\lambda_1,\lambda_2)\in\D^2:
\mbox{det}(\Psi(\lambda_1\lambda_2)-\lambda_2I)=0\}.
\]

\vspace{0.2in}

\textit{Acknowledgement:} The authors are grateful to the referee
for pointing out an error in an earlier version and for numerous
suggestions. The first author's research work is supported by
DST-INSPIRE Faculty Fellowship No. DST/INSPIRE/04/2015/001094. The
research of the second author is supported in part by NBHM (National
Board of Higher Mathematics, India) Research Grant NBHM/R.P.64/2014

\end{document}